\documentclass[12pt,reqno]{amsart}

\usepackage{amsfonts,amsmath,amsthm}
\usepackage{amssymb,epsfig}
\usepackage{cases}
\usepackage[usenames]{color}
\usepackage{enumerate} 

\usepackage{color} 
\definecolor{vert}{rgb}{0,0.6,0}

\theoremstyle{plain}
\newtheorem{thm}{Theorem}[section]

\newtheorem{lem}[thm]{Lemma}

\newtheorem{prop}[thm]{Proposition}
\theoremstyle{remark}
\newtheorem{rem}{\bf{Remark}}
\numberwithin{equation}{section}

\newcommand{\M}{\mathbb{M}}

\newcommand{\R}{\mathbb{R}}

\newcommand{\T}{\mathbb{T}}
\newcommand{\Z}{\mathbb{Z}}

\newcommand{\cC}{\mathcal{C}}

\newcommand{\Li}{L^{\infty}}
\newcommand{\Lip}{{\rm Lip\,}}

\newcommand{\Q}{\mathbb{T}^{n}\times(0,\infty)}

\newcommand{\cQ}{\mathbb{T}^{n}\times[0,\infty)}

\newcommand{\al}{\alpha}

\newcommand{\del}{\delta}
\newcommand{\ep}{\varepsilon}

\newcommand{\sig}{\sigma}

\newcommand{\Del}{\Delta}

\newcommand{\ol}{\overline}

\newcommand{\pl}{\partial}

\newcommand{\diag}{{\rm diag}\,}

\newcommand{\Div}{{\rm div}\,}

\newcommand{\tr}{{\rm tr}\,}

\begin{document}
\title[A new method for large time behavior of Hamilton--Jacobi equations]
{A new method for large time behavior\\ 
of degenerate viscous Hamilton--Jacobi equations\\ 
with convex Hamiltonians}

\author[F. CAGNETTI, D. GOMES, H. MITAKE, H. V. TRAN]
{Filippo Cagnetti, Diogo Gomes, Hiroyoshi Mitake and Hung V. Tran}

\address[F. Cagnetti]
{Center for Mathematical Analysis, Geometry, and Dynamical Systems, Departamento de Matem\'atica, 
Instituto Superior T\'ecnico, 1049-001 Lisboa, Portugal}
\email{cagnetti@math.ist.utl.pt}

\address[D. Gomes]
{Center for Mathematical Analysis, Geometry, and Dynamical Systems, Departamento de Matem\'atica, 
Instituto Superior T\'ecnico, 1049-001 Lisboa, Portugal
and King Abdullah University of Science and Technology (KAUST), 
CSMSE Division, Thuwal 23955-6900, Saudi Arabia}
\email{dgomes@math.ist.utl.pt}

\address[H. Mitake]
{
Department of Applied Mathematics, 
Faculty of Science, Fukuoka University, 
Fukuoka 814-0180, Japan}
\email{mitake@math.sci.fukuoka-u.ac.jp}

\address[H. V. Tran]
{
Department of Mathematics, 
The University of Chicago, 5734 S. University Avenue, Chicago, Illinois 60637, USA}
\email{hung@math.uchicago.edu}


\maketitle

\begin{abstract}
We investigate large-time asymptotics for viscous Hamilton--Jacobi equations with possibly degenerate diffusion terms. 
We establish new results on the convergence, which are the first general ones   concerning equations which are neither uniformly parabolic  nor first order. 
Our method is based on the nonlinear adjoint method and the derivation of new estimates on long time averaging effects. It also extends to the case of weakly coupled systems. 
\end{abstract}


\tableofcontents


\section{Introduction}
In this paper we obtain new results on the study of the large time behavior of Hamilton--Jacobi equations with possibly degenerate diffusion terms \begin{equation}\label{eq:general}
u_t + H(x,Du)  = \tr\big(A(x)D^2u\big)  \qquad \text{ in } \Q,
\end{equation}
where $\T^n$ is the $n$-dimensional torus $\R^n/\Z^n$.
Here $Du, D^2 u$ are the (spatial) gradient and Hessian of the real-valued unknown function $u$ defined on $\T^n \times [0,\infty)$.
The functions $H:\T^n\times\R^n\to\R$ and $A:\T^n\to \M^{n\times n}_{\text{sym}}$ are the Hamiltonian and the diffusion matrix, respectively,
where $\M^{n\times n}_{\text{sym}}$ is the set of $n \times n$ real symmetric matrices. The basic hypotheses that we require are that
$H$ is uniformly convex in the second variable, and $A$ is nonnegative definite.

Our goal in this paper is to study the large time behavior of viscosity solutions of \eqref{eq:general}.  
Namely, we prove that  
\begin{equation}\label{lt}
\|u(\cdot,t)-(v-ct)\|_{\Li(\T^n)} \to 0 \quad \text{as} \ t \to \infty,
\end{equation}
where $(v,c)$ is a solution of the \textit{ergodic problem} 
\begin{equation}\label{eq:ergodic}
H(x, Dv)=\tr\big(A(x)D^2v\big)+c \quad \text{in}\  \T^n. 
\end{equation}
In view of the quadratic or superquadratic growth of the Hamiltonian,
there exists a unique constant $c\in \R$ such that 
\eqref{eq:ergodic} holds true for some $v \in C(\T^n)$ in the viscosity sense.
We notice that in the uniformly parabolic case ($A$ is positive definite), 
$v$ is unique up to additive constants. 
It is however typically the case that $v$ is not unique even up to additive constants when $A$ is degenerate, which makes the convergence \eqref{lt} delicate and hard to be achieved. 
We will state clearly the existence result of \eqref{eq:ergodic}, which itself is important, in Section 2.

It is worth emphasizing here that the study of the large-time asymptotics for 
this type of equations was only available in the literature for the uniformly parabolic case and for the first order case. There was no results on the large-time asymptotics for \eqref{eq:general} with possibly degenerate diffusion terms up to now as far as the authors know.

In the last decade,
a number of authors have studied extensively
the large time behavior of solutions of 
(first order) Hamilton--Jacobi equations 
(i.e., \eqref{eq:general} with $A \equiv 0$), where $H$ is coercive. 
Several convergence results have been established. 
The first general theorem in this direction was proven by
Namah and Roquejoffre in \cite{NR}, under the assumptions:
$p \mapsto H(x,p)$ is convex, $H(x,p)\ge H(x,0)$ for all $(x,p)\in\T^n\times\R^{n}$,   
and $\max_{x\in\T^n}H(x,0)=0$. 
Fathi then gave a breakthrough in this area in  \cite{F2} 
by using a dynamical systems approach from the weak KAM theory.
Contrary to \cite{NR}, the results of \cite{F2} use uniform convexity and smoothness assumptions on the Hamiltonian but do not require any 
condition on the structure above. 
These rely on a deep understanding of the dynamical structure of the solutions and of the corresponding ergodic problem. 
See also the paper of Fathi and Siconolfi \cite{FS} for a beautiful characterization of the Aubry set.
Afterwards, Davini and Siconolfi in \cite{DS} and Ishii in \cite{I2008} refined 
and generalized the approach of Fathi, and studied the asymptotic problem for Hamilton--Jacobi equations 
on $\T^n$ and on the whole $n$-dimensional Euclidean space, respectively.
Besides, Barles and Souganidis \cite{BS} 
obtained additional results, for possibly non-convex Hamiltonians,
by using a PDE method in the context of viscosity solutions. 
Barles, Ishii and Mitake \cite{BIM2} simplified the ideas in \cite{BS} and 
presented the most general assumptions (up to now). 
In general, these methods are based crucially on 
delicate stability results of extremal curves in the context of the dynamical approach in light of the finite speed of propagation, 
and of solutions for time large in the context of the PDE approach. 
It is also important to point out that the PDE approach in \cite{BS, BIM2} 
does not work with the presence of any second order terms.

In the uniformly parabolic setting (i.e., $A$ uniformly positive definite),
Barles and Souganidis \cite{BS2} proved the long-time convergence of solutions. 
Their proof relies on a completely distinct set of ideas from the ones used 
in the first order case as the associated ergodic problem
has a simpler structure.
Indeed, the strong maximum principle holds, the ergodic problem 
has a unique solution up to constants. 
The proof for the large-time convergence in \cite{BS2} strongly depends on this fact.

It is clear that all the methods aforementioned (for both the cases $A\equiv 0$ and $A$ uniformly positive definite) are not applicable for the general
degenerate viscous cases because of 
the presence of the second order terms and the lack of both the finite speed
of propagation as well as the strong comparison principle. 
We briefly describe the key ideas on establishing \eqref{lt} in subsection 1.1.
Here the nonlinear adjoint method, which was introduced by Evans in \cite{Ev1},
plays the essential role in our analysis. 
Our main results are stated in subsection 1.2.

\subsection{Key Ideas}\label{intro:machinery}
Let us now briefly describe the key ideas on establishing \eqref{lt}.
Without loss of generality, we may assume the ergodic constant is $0$ henceforth. 
In order to understand the limit as $t \to \infty$, we introduce a rescaled problem.
For $\ep>0$, set $u^\ep(x,t)=u(x,t/\ep)$. 
Then $(u^\ep)_t (x,t)=\ep^{-1} u_t(x,t/\ep)$, 
$Du^\ep(x,t)=Du(x,t/\ep)$, and $u^\ep$ solves
\begin{equation} \notag
\begin{cases}
\varepsilon u^{\ep}_t + H(x,Du^{\ep})=\tr\big(A(x)D^2u^{\ep}\big) & \text{ in } \T^n\times (0,\infty), \\
u^{\ep} (x,0)=u_0(x),  & \text{ on } \T^n.
\end{cases}
\end{equation}
By this rescaling, $u^\ep(x,1)=u(x,1/\ep)$ 
and we can easily see that to prove \eqref{lt} is equivalent to prove that 
\begin{equation}\notag
\|u^{\ep}(\cdot,1)- v\|_{\Li(\T^n)}\to 0 \quad \text{as} \ \ep\to 0. 
\end{equation}

To show the above, 
we first introduce the following approximation: 
\begin{equation} \notag
\begin{cases}
\varepsilon w^{\ep}_t + H(x,Dw^{\ep})=\tr\big(A(x)D^2w^{\ep}\big)
+\ep^4\Del w^{\ep} & \text{ in } \T^n\times (0,\infty), \\
u^{\ep} (x,0)=u_0(x),  & \text{ on } \T^n.
\end{cases}
\end{equation}
Then, we observe that $w^\ep$ is smooth, and
$$
\|w^\ep(\cdot,1)-u^\ep(\cdot,1)\|_{\Li(\T^n)} \to 0
\quad \text{as} \ \ep \to 0.
$$
It is thus enough to derive the convergence of $w^\ep(\cdot,1)$ as $\ep \to 0$.  
To prove this, we show that
\begin{equation} \label{act-prin}
\ep\|w^{\ep}_{t}(\cdot,1)\|_{\Li(\T^n)} \to 0 \quad  \text{as} \ \ep\to 0, 
\end{equation}
which is a way to prove the convergence \eqref{lt}.  
Indeed, \eqref{lt} is a straightforward consequence of \eqref{act-prin} by using the stability of viscosity solutions. 
We notice that this principle appears in 
the papers of Fathi \cite{F2}, Barles and Souganidis \cite{BS} in the case $A\equiv 0$ in a completely different way. 
More precisely, Barles and Souganidis  \cite{BS} first realized the importance of \eqref{act-prin}, and they gave a beautiful proof of the fact that $\max\{u_t,0\}, \min\{u_t,0\}\to 0$ as $t\to\infty$ in the viscosity sense. In view of this fact, they succeeded to deal with some cases of non-convex Hamilton--Jacobi equations. 
On the other hand, we emphasize that the proofs in \cite{F2, BS} do not work at all for the second order cases, and therefore, one cannot apply it to \eqref{eq:general}. 
One of our key contributions in this  paper is the establishment of \eqref{act-prin} in the general setting.

In order to prove \eqref{act-prin}, we use the nonlinear adjoint method  introduced  by Evans \cite{Ev1}
and give new ingredients on the averaging action as $t \to \infty$ (or equivalently as $\ep \to 0$ by rescaling), as clarified below.
Let $L_{w^\ep}$ be the {\it formal linearized operator} of the regularized equation 
around $w^\ep$, i.e., 
$$
L_{w^\ep} f  := \frac{\partial}{\partial \eta} 
\Big[ \ep (w^{\ep} + \eta f)_t + H(x,Dw^{\ep} + \eta Df)
-\tr\big(A(x)(D^2w^{\ep}+\eta D^2f)\big) 
- \ep^{4}\Delta ( w^{\ep} + \eta f)
\Big] \Big|_{\eta = 0} ,
$$
for any $f\in C^2(\T^n \times(0,\infty))$.
Then we consider the following adjoint equation:  
\begin{equation} \notag
\begin{cases}
L_{w^\ep}^{*} \sig^\ep=0 & \text{ in } \T^n\times (0,1), \\
\sig^{\ep} (x,1)=\del_{x_0}  & \text{ on } \T^n,
\end{cases}
\end{equation}
where $L_{w^\ep}^{*}$ is the formal adjoint operator of $L_{w^\ep}$,
and $\del_{x_0}$ is the Dirac delta measure at some point $x_0 \in \T^n$.
We then see that $\sig^\ep(\cdot,t)$ is a probability measure for all $t \in (0,1)$
and conservation of energy holds, namely,  
$$
\dfrac{d}{dt} \int_{\T^n} 
\big[H(x,Dw^\ep)-\tr\big(A(x)D^2w^\ep\big)-\ep^4 \Del w^\ep\big]\sig^\ep(x,t)\,dx=0.
$$
The conservation of energy in particular gives us 
a different and completely new way to interpret $\ep w^{\ep}_t(\cdot,1)$ as
\begin{equation}\label{con-e}
\ep w^{\ep}_t(x_0,1)=\int_0^1 \int_{\T^n}  
\big[H(x,Dw^\ep)-\tr\big(A(x)D^2w^\ep\big)-\ep^4 \Del w^\ep\big]\sig^\ep(x,t)\,dx\,dt.
\end{equation}
The most important part of the paper is then about showing that
 the right hand side of \eqref{con-e} vanishes as $\ep \to 0$, 
which requires new ideas and estimates (see Lemmas \ref{LEM} and \ref{lem:couple}). 
We also notice that the averaging action appears implicitly in \eqref{con-e} and 
plays the key role here.
More precisely, if we rescale the above integral back to its actual scale, 
it turns out to be 
\begin{equation}\label{average-action}
\dfrac{1}{T} \int_0^T \int _{\T^n}
\left[H(x,Dw^\ep)-\tr\big(A(x)D^2w^\ep\big)-\ep^4 \Del w^\ep\right]\sig^\ep(x,t)\,dx\,dt, 
\end{equation}
where $T=1/\ep \to \infty$.

The nonlinear adjoint method for Hamilton-Jacobi equations was introduced  
by Evans \cite{Ev1} to study
the  vanishing viscosity process, and gradient shock structures of viscosity solutions
 of non convex Hamilton--Jacobi equations. Afterwards, Tran \cite{T1} used it to 
 establish a rate of convergence for static Hamilton--Jacobi equations and was able to relax the convexity assumption of
the Hamiltonians in some cases. Cagnetti, Gomes and Tran \cite{CGT1} then used it to study 
the Aubry--Mather theory in the non convex settings and established the existence
of Mather measures. See also \cite{CGT2, Ev2} for further developments of this new method in the context of Hamilton--Jacobi equations.
We notice further that $\sig^\ep$ here is strongly related to the Mather measures in the context of the weak KAM theory. See \cite{CGT1} for more details.

\subsection{Main Results}

We state the assumptions we use \textit{throughout} the paper 
as well as our main theorems. 

For some given $\theta, C>0$,
we denote by $\mathcal{C}(\theta,C)$ the class of all pairs of $(H,A)$ satisfying
\begin{itemize}
\item[(H1)] $H \in C^2(\T^n \times \R^n)$, and 
$D^2_{pp}H \ge 2\theta I_n$, where $I_n$ is the identity matrix of size $n$,
\item[(H2)] 
$|D_x H(x,p)| \le C(1+|p|^2)$,  
\item[{(H3)}] 
$A(x)=(a^{ij}(x))\in \M^{n\times n}_{\text{sym}}$ with $A(x)\ge 0$,  and $A\in C^2(\T^n)$.
\end{itemize}

\begin{thm}[Main Theorem 1]\label{main-thm1}
Assume that $(H,A) \in \mathcal{C}(\theta,C)$. 
Let $u$ be the solution of {\rm\eqref{eq:general}} with initial data $u(\cdot,0)=u_0 \in C(\T^n)$. 
Then there exists $(v,c) \in C(\T^n) \times \R$ such that 
\eqref{lt} holds, 
where the pair $(v,c)$ is a solution of the ergodic problem \eqref{eq:ergodic}.  
\end{thm}

We also consider the weakly coupled system 
\begin{equation}\label{eq:general-system}
(u_i)_t + H_i(x,Du_i) + \displaystyle\sum_{j=1}^m c_{ij} u_j=\tr\big(A_i(x)D^2u_i\big) \quad \text{ in } \Q, \ \text{for} \ i=1,\ldots,m. 
\end{equation}
where  $(H_i,A_i)\in \mathcal{C}(\theta,C)$ is, for each $i$, the Hamiltonian and diffusion matrix $H_i:\T^n\times\R^n\to\R$ and $A_i:\T^n\to \M^{n\times n}_{\text{sym}}$ and $u_i$ are the real-valued unknown functions on $\T^n \times [0,\infty)$ for $i=1,\ldots,m$. 
The coefficients $c_{ij}$ are given constants for $1\le i, j \le m$ which are assumed to satisfy 
\begin{itemize}
\item[{\rm(H4)}] \ 
$c_{ii}>0$, $c_{ij} \le 0$ for $i \ne j$, $\displaystyle\sum_{j=1}^m c_{ij}=0$ for 
any $i=1,\ldots,m$.  
\end{itemize}
We remark that {\rm(H4)} ensures that \eqref{eq:general-system} is a monotone system. 

Under these conditions, we prove 
\begin{thm}[Main Theorem 2]\label{mainsist}
Assume that $(H_i,A_i) \in \mathcal{C}(\theta,C)$, 
and $c_{ij}$ satisfies 
{\rm(H4)} for $1\le i, j \le m$. 
Let $(u_1,\ldots,u_m)$ be the solution of  \eqref{eq:general-system} with 
initial data
$(u_{01},\ldots, u_{0m}) \in C (\T^n)^m$. 
There exists $(v_1,\ldots,v_m,c) \in C (\T^n)^m\times\R$
such that 
\begin{equation}\label{lt:system}
\|u_i (\cdot, t)-(v_i-ct)\|_{\Li(\T^n)}\to0 \quad \text{as} \ t \to + \infty, \quad \text{for} \  i=1,\ldots m, 
\end{equation}
where $(v_1,\ldots,v_m, c)$ is a solution of the {\it ergodic problem} 
for systems: 
\begin{equation}\notag 
H_i(x,Dv_i) + \displaystyle\sum_{j=1}^m c_{ij} v_j=\tr\big(A_i(x)D^2v_i\big)+c \ \text{ in } \T^n, \ \text{for} \ i=1,\ldots,m. 
\end{equation}
\end{thm}

The study of the large-time behavior of solutions to the weakly coupled system \eqref{eq:general-system} of first order cases (i.e. $A_i\equiv 0$ for $1\le i\le m$) was started by \cite{MT1} and \cite{CLLN} 
independently under  rather restrictive assumptions for Hamiltonians. 
Recently, Mitake and Tran \cite{MT3} were able to establish
convergent results  under rather general assumptions on Hamiltonians. 
Their proof is based on the dynamical approach, inspired by the papers \cite{DS, I2008}, together with a new representation formula for the solutions. 
See \cite{N1} for a PDE approach inspired by \cite{BS}.
We also refer to \cite{MT2} for a related work on  homogenization of
weakly coupled systems of Hamilton--Jacobi equations with fast switching rates.

We prove Theorem \ref{mainsist} by following the ideas described 
above. We notice that the coupling terms cause some additional difficulties
and are needed to be handled carefully. 
We in fact establish a new estimate (see  Lemma \ref{lem:couple} (ii) and Subsection \ref{general}) to control 
these coupling terms in order to achieve \eqref{act-prin}.
This is completely different from the single case. 

After this paper was completed, we learnt that Ley and Nguyen \cite{LN} obtained recently related convergence results for some specific degenerate parabolic equations. They however assume rather restrictive and technical conditions on the degenerate diffusions so that they could combine the PDE approaches in \cite{BS} and \cite{BS2} to achieve the results. 
On the other hand, they can deal with a type of fully nonlinear case, which is not included in ours.

\medskip
This paper is organized as follows: 
Sections \ref{deg-HJ} and \ref{sys-HJ} are devoted 
to prove Theorems \ref{main-thm1} and \ref{mainsist}
respectively.
 We provide details and explanations to the method described above
in subsections \ref{sec-r}, \ref{sec3}, \ref{sec6}, and \ref{sec8} and give 
key estimates in subsections \ref{sec4} 
and \ref{sec7}.

 
\section{Degenerate Viscous Hamilton--Jacobi equations} \label{deg-HJ}
In this section we study degenerate viscous Hamilton--Jacobi equations.
To keep the formulation as simple as possible, 
we first consider the equation
\begin{equation} \notag
{\rm (C)} \qquad 
\begin{cases}
u_t + H(x,Du)  = a(x) \Del u & \text{ in } \Q, \\
u(x,0)=u_{0}(x) & \text{ on } \T^n,
\end{cases}
\end{equation}
where $u_0 \in C(\T^n)$. 
Throughout this section we \textit{always} assume that 
$(H,a I_n)\in\cC(\theta, C)$, i.e., $a$ is supposed to be in $C^2(\T^n)$ with $a\ge0$. 
We remark that all the results proved for this particular case
hold with trivial modifications for the general elliptic operator $\tr(A(x) D^2u)$, except estimate \eqref{esti-1} and Lemma \ref{LEM} (ii).
The corresponding results will be considered in the 
end of this section.

The next three propositions concern 
basic existence results, both for (C) and for the associated stationary problem. 
The proofs are standard, hence omitted.
We refer the readers to the companion paper \cite{MT4} by Mitake and Tran
for the detailed proofs of Propositions \ref{am2} and \ref{thm-2}.

\begin{prop} \label{am}
Let $u_0 \in C(\T^n)$. 
There exists a unique solution $u$ 
of \textnormal{(C)} which is uniformly continuous on $\cQ$.
Furthermore, if $u_0\in\Lip(\T^n)$, then $u\in\Lip(\cQ)$. 
\end{prop}

\begin{prop} \label{am2}
There exists a unique constant $c \in \R$ such that the ergodic problem 
\begin{equation*}
{\rm(E)}\qquad
H(x,Dv(x))=a(x) \Del v +c \quad \text{ in }  \T^{n},
\end{equation*} 
admits a solution $v \in \Lip(\T^n)$.
We call $c$ the ergodic constant of {\rm (E)}.
\end{prop}
In order to get the existence of the solutions of the ergodic problem (E), 
we use the vanishing viscosity method as described by the following Proposition.

\begin{prop}\label{thm-2}
For every $\varepsilon \in (0,1)$ there exists a unique constant $\ol{H}_\ep$ 
such that the following ergodic problem: 
\begin{equation*} 
{\rm(E)}_{\varepsilon}\qquad
H(x,Dv^\ep)=(\ep^4+a(x)) \Del v^\ep +\ol{H}_\ep \quad \textnormal{in} \ \T^n
\end{equation*}
has a unique solution $v^\ep\in\Lip(\T^n)$ up to some additive constants. In addition,
\begin{equation}\label{bound-5}
|\ol{H}_\ep-c| \le C \ep^2, \qquad \| D v^\ep\|_{L^\infty(\T^n)} \le C,
\end{equation}
for some positive constant $C$ independent of $\ep$. Here $c$ is the ergodic constant of {\rm (E)}.
\end{prop}
In view of the quadratic or superquadratic growth of the Hamiltonian $H$, 
we can get \eqref{bound-5} by 
the Bernstein method. See the proof of \cite[Proposition 1.1]{MT4} for details.
By passing to some subsequences if necessary,
$v^\ep-v^{\ep}(x_0)$ for a fixed $x_0\in\T^n$  
converges uniformly to a Lipschitz function $v: \T^n \to \R$
which is a solution of (E) as $\ep \to 0$.

We prove Theorem \ref{main-thm1} in a sequence of subsections by using the 
method described in Introduction.

\subsection{Regularizing Process and Proof of Theorem {\rm\ref{main-thm1}}} \label{sec-r}
We only need to study the case where
$c=0$, where $c$ is the ergodic constant, 
by replacing, if necessary, $H$ and $u(x,t)$ by $H-c$ and $u(x, t)+ct$, respectively.
Therefore, from now on, we always assume that $c = 0$ in this section. 
Also, without loss of generality we can assume that $u_0\in\Lip(\T^n)$, 
since the general case $u_0 \in C(\T^n)$ can be obtained by a standard approximation argument.
In particular, thanks to Proposition \ref{am}, we see that $u$ is Lipschitz continuous on $\cQ$.

As stated in Introduction, we consider a rescaled problem.
Setting $u^\ep(x,t)=u(x,t/\ep)$ for $\ep>0$, where $u$ is the solution of (C), 
one can easily check that $u^{\ep}$ satisfies 
\begin{equation*}
{\rm (C)_\ep}\qquad
\begin{cases}
\ep u^\ep_t + H(x,Du^\ep)  = a(x)\Del u^\ep & \text{ in }\Q, \\
u^\ep(x,0)=u_{0}(x) & \text{ on } \T^n. 
\end{cases}
\end{equation*}
Notice however that in this way we do not have a priori uniform Lipschitz estimates on $\ep$, 
since the Lipschitz bounds on $u$ give us that
\begin{equation}\label{bound-2}
\|u^\ep_t\|_{L^\infty(\T^n \times [0,1])} \le C/\ep, 
\qquad 
\|Du^\ep\|_{L^\infty(\T^n \times [0,1])} \le C.
\end{equation}
In general, the function $u^\ep$ is only Lipschitz continuous.
For this reason, we add a viscosity term to (C)$_\ep$, 
and we consider the regularized equation
\begin{equation}\notag
{\rm(A)_{\ep}^{\eta}}\qquad
\begin{cases}
\ep w^{\varepsilon, \eta}_t + H(x,Dw^{\varepsilon, \eta})  
= (a(x) + \eta ) \Delta w^{\varepsilon, \eta} & \text{ in }\Q,  \\
w^{\varepsilon, \eta}(x,0)=u_{0}(x)  & \text{ on } \T^{n}, 
\end{cases}
\end{equation} 
for $\eta>0$.  
The advantage of considering (A)$_{\ep}^{\eta}$
lies in the fact that the solution is smooth, 
and this will allow us to use the nonlinear adjoint method. 
The adjoint equation corresponding to (A)$_{\ep}^{\eta}$ is
\begin{equation} \notag
{\rm (AJ)_\ep^{\eta}} \qquad 
\begin{cases}
-\ep \sig^{\ep,\eta}_t -\text{div}(D_p H(x,Dw^{\ep,\eta}) \sig^{\ep,\eta})  
=\Del \big((a(x)+\eta)\sig^{\ep,\eta}\big) & \text{ in }\T^n \times (0,1),  \\
\sig^{\ep,\eta}(x,1)=\del_{x_0}  & \text{ on } \T^{n},
\end{cases}
\end{equation} 
where $\del_{x_0}$ is the Dirac delta measure at some point $x_0 \in \T^n$.
\begin{lem}[Elementary Properties of $\sig^{\ep,\eta}$]
We have $\sig^{\ep,\eta} \ge 0$ and
$$
\int_{\T^n} \sig^{\ep,\eta}(x,t)\,dx=1 \quad \text{for all} \ t \in [0,1].
$$
\end{lem}
This is a straightforward result of adjoint operator and easy to check.
Heuristically, the vanishing viscosity method gives 
that the rate of convergence of $w^{\ep,\eta}$ to $w^{\ep}$ as $\eta\to0$ is 
\[\sqrt{\text{viscosity coefficient}}/(\text{the coefficient of} \ w_{t}^{\ep,\eta})\] 
and therefore, 
we naturally expect that we need to choose $\eta=\ep^{\al}$ with $\al>2$. 
We mostly choose $\eta=\ep^4$ hereinafter, and therefore 
we specially denote (A)$_{\ep}^{\ep^4}$, (AJ)$_{\ep}^{\ep^4}$ 
by (A)$_{\ep}$,  (AJ)$_{\ep}$
and $w^{\ep, \ep^4}$, $\sig^{\ep,\ep^4}$ by $w^\varepsilon$, $\sig^{\ep}$.

The proof of Theorem \ref{main-thm1} is based 
on  the two following results, Proposition \ref{appro-lem} and Theorem \ref{thm-1}.
\begin{prop}\label{appro-lem}
Let $w^{\ep}$ be the solution of \textnormal{(A)$_{\ep}$}. 
There exists $C>0$ independent of $\ep$ such that
$$
\|w^\ep (\cdot,1) \|_{C^{1}(\T^n)} \leq C, \quad
\|u^\ep(\cdot,1)-w^\ep(\cdot,1)\|_{L^\infty(\T^n)} \le C\ep.  
$$
\end{prop}
The proof of Proposition \ref{appro-lem} can be derived using standard arguments.
Nevertheless, we give the proof below, since some of the estimates involved will be used later.

Before starting the proof, 
we state a basic property of the function $a\in C^{2}(\T^n)$.
\begin{lem} \label{lem-a}
There exists a constant $C>0$ such that
\begin{equation}\label{a-a}
|Da(x)|^2 \le C a(x) \quad \text{for all} \ x \in \T^n.
\end{equation}
\end{lem}
\begin{proof}
In view of \cite[Theorem 5.2.3]{SV}, $a^{1/2} \in \Lip(\T^n)$.
It is then immediate to get \eqref{a-a} by noticing that, for $a(x)>0$,
\[
D\left (a^{1/2} \right)(x)=\dfrac{Da(x)}{2a^{1/2}(x)}. \qedhere
\]
\end{proof}

\begin{proof}[Proof of Proposition {\rm \ref{appro-lem}}]
Since we are assuming that the ergodic constant for (E) is $0$ now 
and $H$ is quadratic or superquadratic on $p$, we can easily get the first estimate. 
We only prove the second one.

Let $w^{\ep, \eta}$ be the solution of (A)$_{\ep}^{\eta}$
and set $\varphi(x,t)=|Dw^{\ep, \eta}|^2/2$.  
Then, $\varphi$ satisfies 
$$
\ep\varphi_t + D_p H \cdot D\varphi +D_xH \cdot Dw^{\ep,\eta} 
= (\eta+a)(\Del \varphi  - |D^2 w^{\ep,\eta}|^2)
+(Da\cdot Dw^{\ep,\eta}) \Del w^{\ep,\eta}.
$$
We next notice that for $\del>0$ small enough, in light of Lemma \ref{lem-a},
\begin{equation}\label{rem:point}
a_{x_k} w^{\ep,\eta}_{x_k} \Del w^{\ep,\eta} \le 
C |Da| \cdot |\Del w^{\ep,\eta}|
\le \dfrac{C}{\del} + \del |Da|^2 |D^2 w^{\ep,\eta}|^2
\le C + \dfrac{1}{2} a |D^2 w^{\ep,\eta}|^2. 
\end{equation}
Hence
\begin{equation}\notag
\ep\varphi_t + D_p H \cdot D\varphi +\dfrac{1}{2} (\eta+a)|D^2 w^{\ep,\eta}|^2  
\le (\eta+a) \Del \varphi +C.
\end{equation}
Multiply the above by $\sig^\ep$ and integrate over $[0,1] \times \T^n$ to yield 
\begin{equation}\label{esti-1}
\int_0^1 \int_{\T^n} 
( a(x) + \eta ) |D^2 w^{\ep,\eta}|^2
\sig^{\ep,\eta}\,dx\,dt
\le C
\end{equation}
for some $C>0$.

Note that $w^{\ep,\eta}$ is differentiable with respect to $\eta$ by standard regularity results for elliptic equations.   
Differentiating the equation in (A)$_{\ep}^{\eta}$ with respect to 
$\eta$, we get  
$$
\varepsilon \left(w^{\varepsilon,\eta}_{\eta}\right)_t + D_p H(x,D w^{\varepsilon,\eta}) \cdot D w^{\varepsilon,\eta}_{\eta}
= \Delta w^{\varepsilon,\eta} + ( a(x) + \eta ) \Delta w^{\varepsilon,\eta}_{\eta}, \ 
\text{ in } \T^n, 
$$
where $f_{\eta}$ denotes the derivative of a function $f$ with respect 
to the parameter $\eta$. 
Multiplying the above by $\sigma^{\varepsilon, \eta}$ 
and integrating by parts on $[0,1]\times\T^n$ yield
\begin{align*}
\varepsilon w^{\varepsilon,\eta}_{\eta} (x_0,1) 
= \int_0^1 \int_{\T^n} \Delta w^{\varepsilon,\eta} \sigma^{\varepsilon, \eta} \, dx \, dt,
\end{align*}
where we used the fact that $w^{\varepsilon,\eta}_{\eta} (x,0) \equiv 0$.
Thanks to \eqref{esti-1}, by the H\"{o}lder inequality
$$
\varepsilon | w^{\varepsilon,\eta}_{\eta} (x_0,1) |
\leq C \left( \int_0^1 \int_{\T^n} |D^2 w^{\varepsilon,\eta}|^2 \sigma^{\varepsilon,\eta} \, dx \, dt \right)^{1/2}
\leq \frac{C}{\sqrt{\eta}}.
$$ 
By choosing properly the point $x_0$ we have thus
$$
\| w^{\varepsilon,\eta}_{\eta} (\cdot, 1) \|_{L^{\infty}(\T^n)} \leq \frac{C}{\varepsilon \sqrt{\eta}},
$$
which gives 
$$
\| w^{\varepsilon,\eta} (\cdot, 1) - u^{\varepsilon} (\cdot, 1) \|_{L^{\infty}(\T^n)} 
=\|w^{\ep,\eta}(\cdot,1)-w^{\ep,0}(\cdot,1)\|_{L^{\infty}(\T^n)}
\leq \frac{C\sqrt{\eta}}{\varepsilon}.
$$
Finally, observing that $w^{\varepsilon} = w^{\varepsilon, \varepsilon^4}$,
choosing $\eta = \varepsilon^4$ we get the result.  

\end{proof}

Next theorem gives \eqref{act-prin} in the special case of problem (A)$_{\ep}$. 
\begin{thm}\label{thm-1}
We have
$$
\lim_{\ep \to 0}\ep \|w^\ep_t(\cdot,1)\|_{L^\infty(\T^n)} = 0.
$$
More precisely, there exists a positive constant $C$, independent of $\ep$, such that
$$
\ep \|w^\ep_t(\cdot,1)\|_{L^\infty(\T^n)}=
\| H(\cdot,Dw^\ep(\cdot,1))-(\ep^4+a(\cdot)) \Delta w^\ep(\cdot,1)\|_{L^\infty(\T^n)}
\le C \ep^{1/4}.
$$
\end{thm}

The proof of Theorem \ref{thm-1} is postponed to the end of this section.
We can now give the proof of Theorem \ref{main-thm1}.

\begin{proof}[Proof of Theorem {\rm\ref{main-thm1}}]
By Proposition \ref{appro-lem}, 
we can choose a sequence $\{\ep_m\} \to 0$ such that
$\{w^{\ep_m}(\cdot,1)\}$ converges uniformly to a continuous function $v$.
In view of Theorem \ref{thm-1}, $v$ is a solution of (E), and thus a (time independent) 
solution of the equation in (C)$_\ep$.
We let $t_m=1/\ep_m$ and use Proposition \ref{appro-lem} to deduce that
$$
\|u(\cdot,t_m)- v\|_{\Li(\T^n)}\to0  \text{ as }m\to \infty.
$$
Let us show that the limit does not depend on the sequence $\{ t_m \}_{m \in \mathbb{N}}$.
Now, for any $x\in \T^n$, $t>0$ such that $t_m \le t <t_{m+1}$, we use the comparison principle to yield that
$$
|u(x,t)-v(x)| \le \| u(\cdot, t_m +(t-t_m))-v(\cdot)\|_{L^\infty(\T^n)}
\le  \| u(\cdot, t_m)-v(\cdot)\|_{L^\infty(\T^n)}.
$$
Thus,
$$
\lim_{t \to \infty} |u(x,t)-v(x)| \le \lim_{m\to \infty}  \| u(\cdot, t_m)-v(\cdot)\|_{L^\infty(\T^n)}=0,
$$
which gives the conclusion.
\end{proof}

\subsection{Convergence Mechanisms: Degenerate Equations} \label{sec4}

We show now the following key lemma,
which provides integral bounds on first and second order derivatives 
of the difference $w^\ep-v^\ep$ on the support of $\sigma^\ep$
and $a$, 
where $v^{\ep}$ is a solution of (E)$_{\ep}$, and 
$w^{\ep}$ and $\sig^{\ep}$ are solutions of 
\begin{align*}
{\rm(A)_{\ep}}&\qquad
\begin{cases}
\ep w^{\varepsilon}_t + H(x,Dw^{\varepsilon})  
= (a(x) + \ep^4 ) \Delta w^{\varepsilon} & \text{ in }\Q,  \\
w^{\varepsilon}(x,0)=u_{0}(x)  & \text{ on } \T^{n}, 
\end{cases}\\
{\rm (AJ)_\ep}& \qquad 
\begin{cases}
-\ep \sig^{\ep}_t -\text{div}(D_p H(x,Dw^{\ep}) \sig^{\ep})  
=\Del \big((a(x)+\ep^4)\sig^{\ep}\big) & \text{ in }\T^n \times (0,1),  \\
\sig^{\ep}(x,1)=\del_{x_0} & \text{ on } \T^{n},  
\end{cases}
\end{align*}
respectively.

\begin{lem}[Key Estimates] \label{LEM}
There exists a positive constant $C$, independent of $\ep$, such that
the following hold{\rm:}\\
{\rm(i)}\, 
$\displaystyle{\int_0^1 \int_{\T^n} 
\left( \frac{1}{\ep}|D(w^\ep-v^\ep)|^2
+\ep^7|D^2(w^\ep-v^\ep)|^2 \right) \sig^\ep\,dx\,dt
\le C,}$\\ 
{\rm(ii)}\, 
$\displaystyle\int_{0}^{1} \int_{\T^n} a^2 (x) |D^2(w^\ep-v^\ep)|^2 \sig^{\ep}\,dx\,dt \le C \sqrt{\ep}$.
\end{lem}

\begin{proof}
Subtracting equation (A)$_{\ep}$ from (E)$_\ep$,
thanks to the uniform convexity of $H$, we get
\begin{align*}
0&=\ep (v^\ep-w^\ep)_t + H(x,Dv^\ep)-H(x,Dw^\ep)-(\ep^4+a(x)) \Del (v^\ep-w^\ep)-\ol{H}_\ep\\
&\ge \ep (v^\ep-w^\ep)_t +D_p H(x,Dw^\ep)\cdot D(v^\ep-w^\ep)+\theta | D(v^\ep-w^\ep)|^2 \\
&\hspace{.4cm}-(\ep^4+a(x)) \Del (v^\ep-w^\ep)-\ol{H}_\ep.
\end{align*}
Multiply the above inequality by $\sig^\ep$ and integrate by parts on $[0,1] \times \T^n$
to deduce that
\begin{align*}
&\theta \int_0^1 \int_{\T^n} |D(w^\ep-v^\ep)|^2 \sig^\ep\,dx\,dt
\le \ol{H}_\ep-\int_0^1 \int_{\T^n}\ep((v^{\ep}-w^{\ep})\sig^{\ep})_{t}\,dxdt\\
&\, +
\int_0^1 \int_{\T^n}
\big[\ep\sig^{\ep}_{t}+\Div(D_{p}H(x,Dw^{\ep})\sig^{\ep})
+\Del((\ep^4+a)\sig^{\ep})\big](v-w)\,dxdt
\\
=&\, 
\ol{H}_\ep+\ep \left[ \int_{\T^n} (w^\ep-v^\ep)\sig^\ep\,dx \right]_{t=0}^{t=1}  \\
=&\, \ol{H}_\ep + \ep (w^\ep (x_0,1) - v^\ep (x_0))  
- \ep \int_{\T^n} (u_0 (x) - v^\ep (x)) \sig^\ep (x,0) \,dx  \\
=&\, \ol{H}_\ep + \ep w^\ep (x_0,1)  - \ep \int_{\T^n} (v^\ep (x_0) - v^\ep (x)) \sig^\ep (x,0) \,dx 
- \ep \int_{\T^n} u_0 (x) \sig^\ep (x,0) \,dx  \\
\le&\, \ol{H}_\ep + 
C \ep + C \ep \| Dv^\ep\|_{\Li(\T^n)}
- \ep \int_{\T^n} u_0 (x) \sig^\ep (x,0) \,dx 
\le  C \ep, 
\end{align*}
where we used Propositions \ref{appro-lem}, \ref{thm-2} (recall that we set $c=0$) in the last two inequalities.  
We hence get
\begin{equation} \label{dalih}
\int_0^1 \int_{\T^n} |D(w^\ep-v^\ep)|^2 \sig^\ep\,dx\,dt \le C \ep,
\end{equation}
which is the first part of (i).

Next, subtract (A)$_{\ep}$ from (E)$_\ep$ and differentiate with respect to $x_i$ to get
\begin{multline*}
\ep (v^\ep-w^\ep)_{x_i t} +D_p H(x,Dv^\ep)\cdot Dv^\ep_{x_i} - D_p H(x,Dw^\ep) \cdot Dw^\ep_{x_i}\\
+ H_{x_i}(x,Dv^\ep)-H_{x_i}(x,Dw^\ep)-(\ep^4+a) \Del (v^\ep-w^\ep)_{x_i} 
-a_{x_i}\Del(v^\ep-w^\ep)= 0.
\end{multline*}
Let $\varphi(x,t)=|D(v^\ep-w^\ep)|^2/2$. 
Multiplying the last identity by $(v^\ep-w^\ep)_{x_i}$
and summing up with respect to $i$, we achieve that
\begin{align*}
&\ep\varphi_t + D_p H(x,Dw^\ep) \cdot D\varphi
+ \left[ \Big( D_p H(x,Dv^\ep)-D_p H(x,Dw^\ep) \Big) \cdot Dv^\ep_{x_i} \right]
(v^\ep_{x_i}-w^\ep_{x_i})\\
&+ \Big( D_x H(x,Dv^\ep)-D_x H(x,Dw^\ep) \Big) \cdot D(v^\ep-w^\ep)
-(\ep^4+a(x))\big( \Del \varphi-|D^2(v^\ep-w^\ep)|^2\big) \\
&- \left[ Da\cdot D(v^\ep-w^\ep) \right] \Del (v^\ep-w^\ep)= 0.
\end{align*}
By using various bounds on the above as in the proof of Proposition \ref{appro-lem}, we derive that
\begin{multline}\label{bound-11}
\ep \varphi_t+D_p H(x,Dw^\ep) \cdot D\varphi-(\ep^4 +a(x))\Del \varphi+(\ep^4+a(x)/2) |D^2(v^\ep-w^\ep)|^2\\
\le C+C (|D^2 v^\ep| + 1)|D(v^\ep-w^\ep)|^2.
\end{multline}
The last term in the right hand side of \eqref{bound-11} is a dangerous term.
We now take advantage of 
\eqref{esti-1} and \eqref{dalih} to handle it. 
Using the fact that $\| D v^{\ep} \|_{L^{\infty}}$ and $\| D w^{\ep} \|_{L^{\infty}}$
are bounded, we have
\begin{align}
C |D^2 v^\ep|\  |D(v^\ep-w^\ep)|^2
&\le\, 
C |D^2 (v^\ep-w^\ep)|\  |D(v^\ep-w^\ep)|^2 + C |D^2 w^\ep|\  |D(v^\ep-w^\ep)|^2
\nonumber\\
&\le\, \dfrac{\ep^4}{2}  |D^2(v^\ep-w^\ep)|^2
+\dfrac{C}{\ep^4} |D(v^\ep-w^\ep)|^2 + C|D^2 w^\ep|. \label{bound-12}
\end{align}
Combine \eqref{bound-11} and \eqref{bound-12} to deduce 
\begin{multline}\label{bound-13}
\ep \varphi_t+D_p H(x,Dw^\ep) \cdot D\varphi-(\ep^4+a(x)) \Del \varphi+\dfrac{\ep^4}{2} |D^2(v^\ep-w^\ep)|^2
\\ \le C |D(v^\ep-w^\ep)|^2
+\dfrac{C}{\ep^4} |D(v^\ep-w^\ep)|^2 + C|D^2 w^\ep|. 
\end{multline}
We multiply \eqref{bound-13} by $\sig^\ep$, integrate over $[0,1] \times \T^n$, to yield that, in light of \eqref{esti-1} and \eqref{dalih},
\begin{align*}
&\ep^4 \int_0^1 \int_{\T^n} |D^2(w^\ep-v^\ep)|^2 \sig^\ep\,dx\,dt 
\le C\ep+ \dfrac{C}{\ep^4}\ep+ C \int_0^1 \int_{\T^n} |D^2 w^\ep | \sig^\ep \,dx\,dt\\
\le \ & \dfrac{C}{\ep^3}+C \Big ( \int_0^1 \int_{\T^n} |D^2 w^\ep |^2\sig^\ep \,dx\,dt \Big )^{1/2}
\Big (\int_0^1  \int_{\T^n} \sig^\ep \,dx\,dt \Big )^{1/2}
\le \dfrac{C}{\ep^3}+\dfrac{C}{\ep^2} \le  \dfrac{C}{\ep^3}.
\end{align*}

Finally, we prove (ii).
Setting $\psi(x,t)=a(x)|D(v^\ep-w^\ep)(x,t)|^2/2=a(x)\varphi(x,t)$ and multiplying 
\eqref{bound-11} by $a(x)$ we get
\begin{multline*}
\ep \psi_t + D_pH(x,Dw^\ep)\cdot D\psi - (D_pH(x,Dw^\ep)\cdot Da)\varphi
-(\ep^4+a(x)) \Del \psi \\
+ (\ep^4+a(x))(\Del a \varphi+2 Da \cdot D\varphi)
+a(x)(\ep^4+a(x)/2)|D^2(v^\ep-w^\ep)|^2 \\
\le C a(x) (|D^2 v^\ep| + 1)|D(v^\ep-w^\ep)|^2.
\end{multline*}
We use the facts that $Da, \ \Del a$ are bounded on $\T^n$ to simplify the above 
as follows
\begin{multline}\label{bound-15}
\ep \psi_t + D_pH(x,Dw^\ep)\cdot D\psi-(\ep^4+a(x)) \Del \psi
+a(x)(\ep^4+a(x)/2)|D^2(v^\ep-w^\ep)|^2\\
\le C |D(v^\ep-w^\ep)|^2 - 2 (\ep^4+a(x)) Da \cdot D\varphi + C a(x) |D^2 v^\ep|\  |D(v^\ep-w^\ep)|^2.
\end{multline}
Next, we have to control the last two terms on the right hand side of \eqref{bound-15}.
Observe first that for $\del>0$ small enough
\begin{align}
&2 | (\ep^4+a(x)) Da \cdot D\varphi| \le C (\ep^4+a(x))|Da| \ |D^2(v^\ep-w^\ep)| \ |D(v^\ep-w^\ep)|\nonumber\\
\le&\, {\del} (\ep^4+a(x))|Da|^2|D^2(v^\ep-w^\ep)|^2 +\dfrac{C}{\del} |D(v^\ep-w^\ep)|^2
\nonumber\\
\le&\, \dfrac{1}{8} (\ep^4+a(x)) a(x) |D^2 (v^\ep-w^\ep)|^2 +C |D(v^\ep-w^\ep)|^2, 
\label{bound-16} 
\end{align}
where we used Lemma \ref{lem-a} in the last inequality.
On the other hand, 
\begin{align}
&a(x)|D^2 v^\ep| \ |D(v^\ep-w^\ep)|^2
\nonumber\\
\le&\, 
a(x) |D^2 w^\ep| \ |D(v^\ep-w^\ep)|^2 + a(x) |D^2(v^\ep-w^\ep)|\ |D(v^\ep-w^\ep)|
\nonumber\\
\le&\, 
\sqrt{\ep} a(x) |D^2 w^\ep|^2 + \dfrac{C}{\sqrt{\ep}} |D(v^\ep-w^\ep)|^2
+\dfrac{a(x)^2}{8} |D^2(v^\ep-w^\ep)|^2 + C|D(v^\ep-w^\ep)|^2. \label{bound-17}
\end{align}
We combine \eqref{bound-15}, \eqref{bound-16}, and \eqref{bound-17} to deduce that
\begin{multline*}
\ep \psi_t + D_pH(x,Dw^\ep)\cdot D\psi-(\ep^4+a(x)) \Del \psi
+\dfrac{a(x)^2}{4} |D^2(v^\ep-w^\ep)|^2\\
\le (C +C \ep^{-1/2}) |D(v^\ep-w^\ep)|^2 +\ep^{1/2} a(x) |D^2 w^\ep|^2.
\end{multline*}
We multiply the above inequality by $\sig^\ep$, integrate over $\T^n \times [0,1]$ 
and use \eqref{esti-1} to yield the result.
\end{proof}

\begin{rem}
Let us give some comments on the estimates in Lemma \ref{LEM}. \\
1. 
In case $a \equiv 0$, estimate (i) gives us much better control of $D(w^\ep-v^\ep)$ and
$D^2(w^\ep-v^\ep)$ on the support of $\sig^\ep$. More precisely, a priori estimates only imply that
$D(w^\ep-v^\ep)$ and $\ep^4 \Delta (w^\ep-v^\ep)$ are bounded. By using the adjoint equation,
we can get further formally that $\ep^{-1/2}D(w^\ep-v^\ep)$ and $\ep^{7/2}D^2(w^\ep-v^\ep)$ are
bounded on the support of $\sig^\ep$. We notice that while we need to require  
 the uniform convexity of $H$ to obtain the first term in (i), 
the second term is achieved without any convexity assumption on $H$.
A version of the second term in (i) was first derived by Evans \cite{Ev1}. \\
2. If the equation in (C) is uniformly parabolic, i.e., 
$a(x)>0$ for all $x\in\T^n$, then the second term of (i) is not needed anymore as estimate (ii) is much
stronger.
On the other hand, if $a$ is degenerate, then (ii) only provides estimation of $|D^2(w^\ep-v^\ep)|^2 \sig^\ep$
on the support of $a$ and it is hence essential to use the second term in (i) to control the part where $a=0$.
\end{rem}

\subsection{Averaging Action and Proof of  Theorem {\rm\ref{thm-1}}} \label{sec3}
\begin{lem}[Conservation of Energy] \label{lem-1}
The following hold:\\
{\rm(i)} \ 
$\displaystyle \dfrac{d}{dt} \int_{\T^n} (H(x,Dw^\ep)-(\ep^4+a(x)) \Del w^\ep) \sig^\ep \,dx=0,$ \\
{\rm(ii)} \ 
$\displaystyle 
\ep w^\ep_t(x_0,1)=
\int_{0}^{1} \int_{\T^n} (H(x,Dw^\ep)-(\ep^4+a(x)) \Del w^\ep) \sig^\ep \,dx\,dt.$
\end{lem}
We stress the fact that identity Lemma \ref{lem-1} (ii) is extremely important. 
As stated in Introduction, if we scale back the time, 
the integral in the right hand side becomes \eqref{average-action}, that is
the averaging action as $t \to \infty$.
Relation (ii) together with Lemma \ref{LEM} allow us to conclude the proof of Theorem \ref{thm-1}.
\begin{proof}
We only need to prove (i) as (ii) follows directly from (i). 
This is a straightforward result of adjoint operators and comes from a direct calculation:
\begin{align*}
&\dfrac{d}{dt} \int_{\T^n} (H(x,Dw^\ep)-(\ep^4+a(x)) \Del w^\ep) \sig^\ep \,dx\\
= & \int_{\T^n} (D_pH(x,Dw^\ep) \cdot Dw^\ep_t-(\ep^4+a(x)) \Del w^\ep_t) \sig^\ep \,dx\\
&\,+ \int_{\T^n} (H(x,Dw^\ep)-(\ep^4+a(x)) \Del w^\ep) \sig^\ep_t \,dx\\
=&-\int_{\T^n} \Big (
\Div\big(D_pH(x,Dw^\ep)\sig^\ep\big)+\Del(\ep^4+a(x)) \sig^\ep)\Big ) w^\ep_t \,dx
-\int_{\T^n}\ep w^\ep_t \sig^\ep_t \,dx=0. 
\qedhere
\end{align*}
\end{proof}

We now can give the proof of Theorem \ref{thm-1}, which is the main principle to 
achieve large time asymptotics, by using the averaging action above and the key estimates 
in Lemma \ref{LEM}.

\begin{proof}[Proof of Theorem {\rm\ref{thm-1}}]
Let us first choose $x_0$ such that
\begin{align*}
&|\ep w^\ep_t(x_0,1)|= | H(x_0,Dw^\ep(x_0,1))-(\ep^4+a(x_0)) \Del w^\ep(x_0,1) |  \\
=\, &\| H(\cdot,Dw^\ep(\cdot,1))-(\ep^4+a(\cdot)) \Delta w^\ep(\cdot,1)\|_{L^\infty(\T^n)}. 
\end{align*}
Thanks to Lemma \ref{lem-1} and Proposition \ref{thm-2},
\begin{align*}
&\ep \| w^\ep_t (\cdot,1) \|_{L^{\infty}(\T^n)} = \| H(\cdot,Dw^\ep(\cdot,1))-(\ep^4+a(\cdot)) \Delta w^\ep(\cdot,1)\|_{L^\infty(\T^n)} \\
=&\, \left| \int_{0}^{1} \int_{\T^n} (H(x,Dw^\ep)-(\ep^4+a) \Del w^\ep) \sig^\ep \,dx\,dt \right| \\
\leq&\, 
\int_{0}^{1} \int_{\T^n} 
|(H(x,Dw^\ep)-(\ep^4+a) \Del w^\ep)-(H(x,Dv^\ep)-(\ep^4+a) \Del v^\ep)| \sig^\ep \,dx\,dt + |\ol{H}_\ep| \\
\le&\, 
\int_{0}^{1} \int_{\T^n} \left[ C|D(w^\ep-v^\ep)| 
+( \ep^4+a) |\Del (w^\ep-v^\ep)| \right] \sig^\ep \,dx\,dt + C \ep^2. 
\end{align*}
We finally use the H\"older inequality and Lemma \ref{LEM} to get
\begin{align*}
&\ep \| w^\ep_t (\cdot,1) \|_{L^{\infty}(\T^n)}\\
\leq&\, C \left( \int_0^1 \int_{\T^n} |D(w^\ep-v^\ep)|^2 \sig^\ep\,dx\,dt \right)^{1/2}
+ C \ep^4 \left( \int_0^1 \int_{\T^n} |D^2(w^\ep-v^\ep)|^2 \sig^\ep\,dx\,dt \right)^{1/2}\\
&+ C \left( \int_0^1 \int_{\T^n} a^2 (x)|D^2(w^\ep-v^\ep)|^2 \sig^\ep\,dx\,dt \right)^{1/2} +C \ep^2
\le C \ep^{1/4}. \qedhere
\end{align*}
\end{proof}

\subsection{General Case} 

In this subsection we consider the general case \eqref{eq:general}. 
As pointed out before we only need to address
the analogs to estimate \eqref{esti-1} and Lemma \ref{LEM} (ii). These are
\begin{align}
&\int_0^1\int_{\T^n}
\big[a^{ij}(x)w^{\ep}_{x_ix_k}w^{\ep}_{x_jx_k}+
\ep^{4}|D^{2}w^{\ep}|^2\big]\sig^{\ep}\,dxdt\le C, \label{general-1}\\
&\int_0^1\int_{\T^n}
a^{ij}(x) a^{ll}(x) (v^\ep-w^\ep)_{x_i x_k}  (v^\ep-w^\ep)_{x_j x_k}
\sig^{\ep}\,dxdt\le C\sqrt{\ep}.   \label{general-2}
\end{align}
In the previous formulas, and throughout this section we will use Einstein's convention 
except in a few places where the summation signs are explicitly written to avoid ambiguities. 
The proofs of \eqref{general-1} and \eqref{general-2} follow the same lines as before. 
Recall that $w^\ep$ and $\sig^\ep$ satisfy
\begin{align}
&
\begin{cases}
\ep w^{\varepsilon}_t + H(x,Dw^{\varepsilon})  
= a^{ij}(x)w^{\varepsilon}_{x_ix_j} + \ep^4\Del w^{\varepsilon} & \text{ in }\Q,  \\
w^{\varepsilon}(x,0)=u_{0}(x) & \text{ on } \T^{n}, 
\end{cases} \label{general:approx}
\end{align}
and
\begin{align*}
&
\begin{cases}
-\ep \sig^{\ep}_t -\text{div}(D_p H(x,Dw^{\ep}) \sig^{\ep})  
=\pl_{x_ix_j}\big(a^{ij}(x)\sig^{\ep}\big)+\ep^4\Del\sig^{\ep} & \text{ in }\T^n \times (0,1),  \\
\sig^{\ep}(x,1)=\del_{x_0}  & \text{ on } \T^{n}, 
\end{cases} \notag
\end{align*}
and $v^\varepsilon$ is a solution to the approximate cell problem
\begin{equation}
\label{BBB}
 H(x,Dv^{\varepsilon})  
= a^{ij}(x)v^{\varepsilon}_{x_ix_j} + \ep^4\Del v^{\varepsilon} \quad \text{ in }\ \T^n.
\end{equation}

We need the following estimates, which are from \cite[Lemma 3.2.3]{SV},
\begin{align}
&|Da^{ij}| \le C \left((a^{ii})^{1/2}+(a^{jj})^{1/2}\right)  
&\text{for}\ 1 \leq i,j \leq n,\label{SV-a1}\\
&(\tr(A_{x_k} S))^2 \leq C \tr(SAS) 
&\text{for} \ S \in \M^{n\times n}_{\text{sym}}, \ 1 \leq k \leq n, \label{SV-a2}
\end{align}
for some constant $C$ depending only on $n$ and $\|D^2 A\|_{L^\infty(\T^n)}$.

We address first \eqref{general-1}. To do so, 
as before, setting $\varphi:=|Dw^{\ep}|^2/2$, we obtain
\[
\ep\varphi_t+D_pH(x,Dw^{\ep})\cdot D\varphi\le 
a^{ij}(\varphi_{x_ix_j}-w_{x_ix_k}^\ep w_{x_jx_k}^\ep )+
\ep^4(\Del\varphi-|D^2 w^{\ep}|^2)+a^{ij}_{x_k}w_{x_ix_j}^\ep w_{x_k}^\ep +C.
\]
The key term to estimate is
$a^{ij}_{x_k}w^{\ep}_{x_ix_j}w^{\ep}_{x_k}$, as all the others do not pose any further problem. 
This is done as follows with help of \eqref{SV-a2}:
\begin{align*}
a^{ij}_{x_k}w^{\ep}_{x_ix_j}w^{\ep}_{x_k}
=\tr(A_{x_k}D^2 w^\ep)w^\ep_{x_k} \leq \frac{1}{2} \tr(D^2 w^\ep A D^2 w^\ep)+C=\frac{1}{2}a^{ij} w^\ep_{x_i x_k} w^\ep_{x_j x_k}+C.
\end{align*}
Then the proof follows exactly as before.

Concerning estimate \eqref{general-2},  
 as before we subtract  \eqref{general:approx} from \eqref{BBB} and differentiate with respect to $x_k$ to get
\begin{multline*}
\ep (v^\ep-w^\ep)_{x_k t} +D_p H(x,Dv^\ep)\cdot Dv^\ep_{x_k} - D_p H(x,Dw^\ep) \cdot Dw^\ep_{x_k}\\
+ H_{x_k}(x,Dv^\ep)-H_{x_k}(x,Dw^\ep)-(\ep^4 \delta^{ij}+a^{ij}) ( (v^\ep-w^\ep)_{x_k})_{x_i x_j} 
-a^{ij}_{x_k}(v^\ep-w^\ep)_{x_i x_j}= 0, 
\end{multline*}
where $\delta^{ij}$ is the Kronecker delta. We multiply the previous equation by $a^{ll} (v^\ep-w^\ep)_{x_k}$ and set $\psi=a^{ll}|D(v^\ep-w^\ep)|^2/2$. After 
some tedious computations we conclude that the main order term from which  \eqref{general-2} follows is 
\begin{align}
a^{ij} a^{ll} (v^\ep-w^\ep)_{x_i x_k}  (v^\ep-w^\ep)_{x_j x_k}
&=\tr(A) \tr(D^2(v^\ep-w^\ep) A D^2(v^\ep-w^\ep))\notag\\
&= \left( \sum_{l} d^l \right) \sum_{k,m} \left ( \sum_i \sqrt{d^m} p^{mi} (v^\ep-w^\ep)_{x_i x_k} \right)^2, \label{est222}
\end{align}
where $A$ is diagonalized as $A=P^T D P$ with $D=\diag\{d^1,\ldots,d^n\}$ with $d^i\ge0$,  
and $P^T P= I_n$.
As before, a number of error terms need to be controlled. The procedure is completely analogous, except for 
two error terms which need to be addressed in a slightly different way. These are
$$
a^{ij} a^{ll}_{x_i} (v^\ep-w^\ep)_{x_j x_k} (v^\ep-w^\ep)_{x_k} \quad \text{and} \quad
a^{ij}_{x_k} a^{ll} (v^\ep-w^\ep)_{x_i x_j} (v^\ep-w^\ep)_{x_k}.
$$
The first term is handled in the following way: 
using \eqref{SV-a1} we have
\begin{align*}
&a^{ij} a^{ll}_{x_i} (v^\ep-w^\ep)_{x_j x_k} (v^\ep-w^\ep)_{x_k}
=a^{ll}_{x_i} p^{mi} p^{mj} d^m (v^\ep-w^\ep)_{x_j x_k} (v^\ep-w^\ep)_{x_k} \\ 
\leq &\, 
C \left(\sum_l \sqrt{d^l}\right)\sum_{k,m} d^m \left|  p^{mj} (v^\ep-w^\ep)_{x_j x_k}\right| |D (v^\ep-w^\ep)|\\
\leq &\,
 \frac{1}{4}\left( \sum_{l} d^l \right) \sum_{k,m} \left ( \sum_j \sqrt{d^m} p^{mj} (v^\ep-w^\ep)_{x_j x_k} \right)^2 +  C|D(v^\ep-w^\ep)|^2.
\end{align*}

Concerning the second term, using \eqref{SV-a2} we obtain 
\begin{align*}
&a^{ij}_{x_k} a^{ll} (v^\ep-w^\ep)_{x_i x_j} (v^\ep-w^\ep)_{x_k}
=\tr(A) \tr(A_{x_k} D^2(v^\ep-w^\ep)) (v^\ep-w^\ep)_{x_k}\\
\leq&\, 
\frac{1}{4}\tr(A) \tr(D^2(v^\ep-w^\ep) A D^2(v^\ep-w^\ep))+ C|D(v^\ep-w^\ep)|^2.
\end{align*}
This shows therefore that the two error terms are well controlled by \eqref{est222}.


\section{Weakly Coupled Systems of Hamilton--Jacobi Equations} \label{sys-HJ}

In this section, we prove Theorem \ref{mainsist}. 
Our proof follows along the lines of the scalar case, together with additional estimates 
for the coupling terms. 
To simplify the presentation we start first with the following weakly coupled system of first-order Hamilton--Jacobi equations: 
\begin{equation} \notag
{\rm (SC)}\qquad 
\begin{cases}
(u_1)_t + H_1(x,Du_1) + u_1-u_2  = 0 & \text{ in } \Q, \\
(u_2)_t + H_2(x,Du_2) + u_2-u_1  = 0 & \text{ in } \Q, \\
u_i(x,0)=u_{0i}(x) & \text{ on } \T^n \ \text{for} \ i=1,2. 
\end{cases}
\end{equation}
Throughout this section we \textit{always} 
assume that $u_{0i}\in C(\T^n)$ and that the pairs $(H_i,0)\in\cC(\theta,C)$ 
for $i=1,2$. 

We observe that almost all results we prove for (SC) are valid with trivial modifications for general  weakly coupled systems with possibly degenerate diffusion terms. 
Indeed, if we combine the arguments 
in Section 2 and Section 3 below, 
then we can immediately get the result on the large-time 
behavior for weakly coupled systems of degenerate viscous 
Hamilton--Jacobi equations \eqref{eq:general-system}. 
We present the simplest case here since we want to concentrate on 
the difficulty coming from the coupling terms of the system. 
We derive new estimates for the coupling terms (see part (ii) of Lemma \ref{lem:couple} and Subsection \ref{general}), which help us to control the large time average on the coupling terms and achieve the desired results.


We first state the basic existence results for (SC) and 
for the associated stationary problem. The proofs of the next three propositions are standard, hence omitted.
\begin{prop} \label{qaz1}
Let $(u_{01}, u_{02}) \in C(\T^n)^2$. 
There exists a unique solution $(u_1,u_2)$ of \textnormal{(SC)} 
which is uniformly continuous on $\cQ$. 
Furthermore, if $u_{0i}\in\Lip(\T^n)$, then $u_{i}\in\Lip(\cQ)$ 
for $i=1,2$. 
\end{prop}

We refer to \cite[Section 4]{CGT2} for the following results.
\begin{prop} \label{qaz2}
There exists a unique constant $c \in \R$
such that the ergodic problem{\rm:}
\begin{equation*} 
{\rm (SE)}\qquad
\begin{cases}
H_1(x,Dv_1) + v_1-v_2 = c & \textnormal{ in } \T^n, \\
H_2(x,Dv_2) + v_2-v_1 = c & \textnormal{ in } \T^n, 
\end{cases}
\end{equation*}
has a solution $(v_1,v_2) \in \Lip(\T^n)^2$.
We call $c$ the ergodic constant of {\rm (SE)}.
\end{prop}

\begin{prop} 
For every $\ep > 0$ sufficiently small there exists a unique $\overline{H}_{\ep}\in \R$
such that the following ergodic problem{\rm:} 
\begin{equation*} 
{\rm (SE)_\ep}\qquad
\begin{cases}
 H_1(x,Dv_1^\ep) + v_1^\ep - v_2^\ep 
 = \ep^4 \Delta v^{\ep}_1 + \overline{H}_{\ep} & \textnormal{ in } \T^n, \\
H_2(x,Dv_2^\ep) + v_2^\ep - v_1^\ep 
= \ep^4 \Delta v^{\ep}_2 +\overline{H}_{\ep} & \textnormal{ in } \T^n, 
\end{cases}
\end{equation*}
has a unique solution $(v^\ep_1,v^\ep_2) \in \Lip (\T^n)^2$ up to 
additional constants.
In addition, 
\begin{equation*}
|\ol{H}_\ep-c| \le C \ep^2, \quad \| D v_i^\ep\|_{L^\infty(\T^n)} \le C, \ \text{for} \ i=1,2,
\end{equation*}
for some  positive constant $C$ independent of $\ep$.
Here $c$ is the ergodic constant of {\rm(SE)}. 
\end{prop}
Without loss of generality, we may assume $c=0$ as in Section \ref{deg-HJ}
henceforth.

\subsection{Regularizing Process and Proof of Theorem {\rm\ref{mainsist}}} \label{sec6}
In the following we will assume that 
$(u_1, u_2)$ is Lipschitz on $\cQ$, as it was done in Section \ref{deg-HJ}.
Once again, we will follow the method stated in Introduction.

We perform a change of time scale.  
For $\ep>0$, let us set $u_i^\ep(x,t)=u_i(x,t/\ep)$, which is the solution of  
\begin{equation*}
{\rm (SC)_\ep} \qquad 
\begin{cases}
\ep (u_1^\ep)_t + H_1(x,Du_1^\ep) + u_1^\ep-u_2^\ep  = 0
 & \textnormal{ in } \Q, \\
\ep (u_2^\ep)_t + H_2(x,Du_2^\ep) + u_2^\ep-u_1^\ep  = 0 & \textnormal{ in } \Q, \\
u_i^\ep(x,0)=u_{0i}(x) &\textnormal{ on }\T^n\ \text{for} \ i=1,2,
\end{cases}
\end{equation*}
and we approximate (SC) by adding viscosity terms to the equations: 
\begin{equation*}
{\rm (SA)_{\ep}} \qquad
\begin{cases}
\ep (w_1^\ep)_t + H_1(x,Dw_1^\ep) + w_1^\ep-w_2^\ep  = \ep^4 \Del w_1^\ep
& \textnormal{ in } \Q, \\
\ep (w_2^\ep)_t + H_2(x,Dw_2^\ep) + w_2^\ep-w_1^\ep  = \ep^4 \Del w_2^\ep 
& \textnormal{ in } \Q, \\
w_i^\ep(x,0)=u_{0i}(x) &\textnormal{ on }\T^n \ \text{for} \ i=1,2.
\end{cases}
\end{equation*}
We can conclude the proof of Theorem \ref{mainsist} with the following two results.
\begin{lem}\label{appro-lem2}
Let $(w^{\ep}_1,w^{\ep}_2)$ be the solution of \textnormal{(SA)$_{\ep}$}. 
There exists $C>0$ independent of $\ep$ such that
$\| w^\ep_i (\cdot,1) \|_{C^{1}(\T^n)} \leq C$,  
$\|u_i^\ep(\cdot,1)-w_i^\ep(\cdot,1)\|_{L^\infty(\T^n)} \le C\ep$ 
for $i=1,2$.
\end{lem}

\begin{thm}\label{thm-12}
We have
$$
\lim_{\ep \to 0}
\max_{i=1,2}\ep \|(w_i^\ep)_t(\cdot,1)\|_{L^\infty(\T^n)}=0. 
$$
\end{thm}

\subsection{Convergence Mechanisms: Weakly Coupled Systems} \label{sec7}
The adjoint system corresponding to (SA)$_{\ep}$ is 
\begin{equation*} 
{\rm (SAJ)_\ep}\quad 
\begin{cases}
-\ep (\sig_1^\ep)_t - \text{div}(D_pH_1(x,Dw_1^\ep)\sig_1^\ep) + \sig_1^\ep-\sig_2^\ep
= \ep^4 \Del \sig_1^\ep & \textnormal{ in } \T^n \times (0,1), \\
-\ep (\sig_2^\ep)_t - \text{div}(D_pH_2(x,Dw_2^\ep)\sig_2^\ep) + \sig_2^\ep-\sig_1^\ep  
= \ep^4 \Del \sig_2^\ep & \textnormal{ in } \T^n \times (0,1), \\
\sigma_i^\ep(x,1) = \delta_{ik} \del_{x_0} &\textnormal{ on }\T^n \ \text{for} \ i=1,2.
\end{cases}
\end{equation*}
where $\del_{ik}=1$ if $i=k$ and 
$\del_{ik}=0$ if $i\not=k$, and 
$x_0 \in \T^n$ and $k \in \{1,2\}$ are to be chosen later. 
Notice that for any choice of $k$, either $\sig_1^\ep(\cdot,1)=0$ or $\sig_2^\ep(\cdot,1)=0$.
Let us record some elementary properties of $(\sig_1^\ep,\sig_2^\ep)$ first.
\begin{lem}[Elementary properties of $(\sig_1^\ep,\sig_2^\ep)$]
We have $\sig_i^\ep \ge 0$ for $i=1,2$ and
$$
\sum_{i=1}^2 \int_{\T^n} \sig_i^\ep(x,t)\,dx=1 \quad \text{for all} \ t\in [0,1].
$$
\end{lem}
We next derive key integral bounds for $(w_1^\ep,w_2^\ep)$,
$(v_1^\ep,v_2^\ep)$ and their derivatives on the supports of $(\sig_1^\ep,\sig_2^\ep)$.
\begin{lem}[Key estimates for weakly coupled systems] \label{lem:couple}
The followings hold true:\\
{\rm (i)} \
$\displaystyle
\int_0^1 \int_{\T^n} 
\sum_{i=1}^2 \left( \dfrac{1}{\ep}|D(w_i^\ep-v_i^\ep)|^2
+\ep^7 |D^2(w_i^\ep-v_i^\ep)|^2 \right) \sig_i^\ep\,dx\,dt \le C,$\\
{\rm (ii)} \  
$\displaystyle
\int_0^1 \int_{\T^n}  [(w_1^\ep-v_1^\ep)
-(w_2^\ep-v_2^\ep)]^2 (\sig_1^\ep+\sig_2^\ep)\,dx\,dt \le C\ep$. 
\end{lem}
Lemma \ref{lem:couple} (ii) is a new observation on the study of weakly coupled systems, which gives us the large time average control on the coupling terms. This is actually the key point in the derivation of the main result for systems (Theorem \ref{mainsist} and Theorem \ref{thm-12}) as one can see in the proof of Lemma \ref{coe-sys}.

\begin{proof}
We will only prove (ii),
since part (i) can be derived
by repeating the proof of Lemma \ref{LEM}.

Thanks to Lemma \ref{appro-lem2}, 
we can always add to the pair $(v_1^\ep , v_2^\ep)$
an arbitrarily large constant $C$ (independent of $\ep$)
such that
\begin{equation} \label{v-w}
2C \ge v_i^\ep \ge w_i^\ep \ \text{in}\ \T^n, \ \text{for} \ i=1,2.
\end{equation}

Let $\varphi_i=(v_i^\ep-w_i^\ep)^2/2$ for $i=1,2$.
Subtract the first equation of (SA)$_{\ep}$ from the first equation of (SE)$_\ep$, 
and multiply the result by $v_1^\ep-w_1^\ep$ to get
\begin{align*}
&\ep (v_1^\ep-w_1^\ep)(v_1^\ep-w_1^\ep)_t
+(v_1^\ep-w_1^\ep)(H_1(x,Dv_1^\ep)-H_1(x,Dw_1^\ep))\\
&+ (v_1^\ep-w_1^\ep)^2
-(v_1^\ep-w_1^\ep)(v_2^\ep-w_2^\ep) =
\ep^4 (v_1^\ep-w_1^\ep) \Delta (v_1^\ep-w_1^\ep) +\ol{H}_\ep (v_1^\ep-w_1^\ep).
\end{align*}
We employ the convexity of $H_1$ and \eqref{v-w} to deduce that
\begin{multline}\label{cc-1}
\ep(\varphi_1)_t
+D_pH_1(x,Dw_1^\ep)\cdot D\varphi_1
+ \varphi_1 - \varphi_2\\
+ \frac{1}{2} [ (v_1^\ep-w_1^\ep)-(v_2^\ep-w_2^\ep)]^2
\le
\ep^4  \Delta\varphi_1 - \ep^4 |D(v_1^\ep-w_1^\ep)|^2+C |\ol{H}_\ep|.
\end{multline}
Similarly,
\begin{multline}\label{cc-2}
\ep(\varphi_2)_t
+D_pH_2(x,Dw_2^\ep)\cdot D\varphi_2
 + \varphi_2 - \varphi_1\\
+ \frac{1}{2} [ (v_1^\ep-w_1^\ep)-(v_2^\ep-w_2^\ep)]^2
\le
\ep^4  \Delta\varphi_2 - \ep^4 |D(v_2^\ep-w_2^\ep)|^2+C |\ol{H}_\ep|.
\end{multline}
Multiplying \eqref{cc-1}, \eqref{cc-2} by $\sig_1^\ep,\ \sig_2^\ep$ respectively, and integrating by parts
\begin{align*}
&\frac{1}{2} \int_0^1 \int_{\T^n}  [ (v_1^\ep-w_1^\ep)-(v_2^\ep-w_2^\ep)]^2 (\sig_1^\ep+\sig_2^\ep)\,dx\,dt \\
\leq \  & -\sum_{i=1}^2  \varepsilon \left[ \int_{\T^n} \varphi_i \sigma^{\ep}_i \, dx \right]_{t=0}^{t=1}+
C |\ol{H}_\ep| - \ep^4  \sum_{i=1,2} \int_{\T^n} |D(v_i^\ep-w_i^\ep)|^2 \sigma_i^\ep\, dx
\le C \ep,
\end{align*}
which implies (ii).
\end{proof}

\subsection{Averaging Action and Proof of Theorem \ref{thm-12}} \label{sec8}
For each $i \in \{1,2\}$, setting $j=3-i$ we have $\{i,j\}=\{1,2\}$.
The following result concerning conservation
of energy and averaging action is analogous to Lemma \ref{lem-1} and therefore 
we omit the proof. 
\begin{lem}[Conservation of Energy for weakly coupled systems] \label{coe-sys}
The following hold{\rm:}\\
{\rm(i)} \ 
$\displaystyle \dfrac{d}{dt} \int_{\T^n} \sum_{i=1}^2
(H_i(x,Dw_i^\ep) +w_i^\ep -w_j^\ep-\ep^4 \Del w_i^\ep)\sig_i^\ep\,dx=0.$ 
\begin{align*}
{\rm (ii)} \ 
-(k-1)\ep (w_1^\ep)_t(x_0,1)&-(2-k)\ep (w_2^\ep)_t(x_0,1) \qquad \qquad \qquad \qquad \qquad \qquad \qquad \qquad \\
=&\int_{0}^1  \int_{\T^n} \sum_{i=1}^2
(H_i(x,Dw_i^\ep) +w_i^\ep -w_j^\ep-\ep^4 \Del w_i^\ep)\sig_i^\ep\,dx\,dt, 
\end{align*}
where $k=1, 2$. 
\end{lem}

\begin{proof}[Proof of Theorem {\rm\ref{thm-12}}]
Without loss of generality, we assume that there exists $x_0 \in \T^n$ such that
$$
\ep |(w_1^\ep)_t(x_0,1)|= \ep \max_{i=1,2}\|(w_i^\ep)_t(\cdot,1)\|_{L^\infty(\T^n)}.
$$
We then choose $k=1$ in (SAJ)$_{\ep}$ and use Lemma \ref{coe-sys} to get
\begin{align*}
&\ep \max_{i=1,2}\|(w_i^\ep)_t(\cdot,1)\|_{L^\infty(\T^n)} 
= \left| \int_{0}^1  \int_{\T^n} \sum_{i=1}^2
(H_i(x,Dw_i^\ep) +w_i^\ep -w_j^\ep-\ep^4 \Del w_i^\ep)\sig_i^\ep\,dx\,dt  \right| \\
&\leq \left| \int_{0}^{1} \int_{\T^n} 
\sum_{i=1}^2 \Big\{H_i(x,Dw_i^\ep)+w_i^\ep -w_j^\ep-\ep^4 \Del w_i^\ep) \right. \\
& \hspace{4cm} -(H_i(x,Dv_i^\ep) +v_i^\ep -v_j^\ep-\ep^4 \Del v_i^\ep)\Big\}\sig^\ep_i \,dx\,dt \Bigg| + |\ol{H}_\ep| \\
& \le \int_{0}^{1} \int_{\T^n} \sum_{i=1}^2 \left[ C|D(w^\ep_i-v^\ep_i)| 
+ \ep^4 |\Del (w^\ep_i-v^\ep_i)| \right] \sig^\ep_i \,dx\,dt \\
& \hspace{3.2cm}+\int_0^1 \int_{\T^n} \big| [(w_1^\ep-v_1^\ep)
-(w_2^\ep-v_2^\ep)] (\sig_1^\ep-\sig_2^\ep) \big| \,dx\,dt + |\ol{H}_\ep|. 
\end{align*}
Thus, 
\begin{align*}
&\left| \int_{0}^1  \int_{\T^n} \sum_{i=1}^2
(H_i(x,Dw_i^\ep) +w_i^\ep -w_j^\ep-\ep^4 \Del w_i^\ep)\sig_i^\ep\,dx\,dt  \right| \\
&\leq C \sum_{i=1}^2 \Big\{\left[ \int_0^1 \int_{\T^n} |D(w^\ep_i-v^\ep_i)|^2 \sig^\ep_i\,dx\,dt \right]^{\frac{1}{2}} 
+  \ep^4 \left[ \int_0^1 \int_{\T^n} |D^2(w^\ep_i-v^\ep_i)|^2 \sig^\ep_i\,dx\,dt \right]^{\frac{1}{2}} \Big\}\\
&\hspace{.4cm}+  \sum_{i=1}^2 \left[ \int_0^1 \int_{\T^n}
 | (w_1^\ep-v_1^\ep) -(w_2^\ep-v_2^\ep) |^2 \sig_i^\ep  \,dx\,dt \right]^{1/2}+ |\ol{H}_\ep|
\le C \sqrt{\ep},
\end{align*}
where the last inequality follows by using Lemma \ref{lem:couple}.
\end{proof}

\subsection{General Case} \label{general}
We address now the general case of
systems of $m$-equations 
of the form
$$
(u_i)_t + H_i(x,Du_i) + \sum_{j=1}^m c_{ij} u_j=0 \quad \text{in} \ \Q, 
$$
with $(H_i,0)\in \cC(\theta,C)$ and $c_{ij}$ satisfying (H4) 
for any $1\le i,j\le m$. 
As stated before, the key point is to  generalize
the coupling terms as in part (ii) of Lemma \ref{lem:couple}. 
More precisely, we show that
\begin{equation}\label{m-equ-1}
\lim_{\ep \to 0} \int_0^1 \int_{\T^n} \sum_{j=1}^m  |c_{ij}| \,
 [(w_j^\ep-v_j^\ep)
-(w_i^\ep-v_i^\ep)]^2 \sig_i^\ep\,dx\,dt=0.
\end{equation}
Set $\varphi_i=(v_i^\ep-w_i^\ep)^2/2$ for $i=1,\ldots,m$. Then
we can compute that
\begin{align*}
&\ep (v_i^\ep-w_i^\ep)(v_i^\ep-w_i^\ep)_t
+(v_i^\ep-w_i^\ep)(H_i(x,Dv_i^\ep)-H_i(x,Dw_i^\ep))\\
&+  \sum_{j=1}^m c_{ij} (v_i^\ep-w_i^\ep)(v_j^\ep-w_j^\ep) =
\ep^4 (v_i^\ep-w_i^\ep) \Delta (v_i^\ep-w_i^\ep) +\ol{H}_\ep (v_i^\ep-w_i^\ep).
\end{align*}
The last term in the right hand side of the above identity can be written as
\begin{align*}
& \sum_{j=1}^m c_{ij} (v_i^\ep-w_i^\ep)(v_j^\ep-w_j^\ep) 
=\sum_{j \ne i} |c_{ij}| \, \Big\{  (v_i^\ep-w_i^\ep)^2-(v_i^\ep-w_i^\ep)(v_j^\ep-w_j^\ep) \Big \}\\
=&\sum_{j \ne i} |c_{ij}| \, \Big \{  \frac{1}{2}(v_i^\ep-w_i^\ep)^2
-\frac{1}{2}(v_j^\ep-w_j^\ep)^2+\frac{1}{2}[(v_i^\ep-w_i^\ep)-(v_j^\ep-w_j^\ep)]^2 \Big \}\\
=&\sum_{j=1}^m c_{ij} \varphi_j+\dfrac{1}{2}\sum_{j=1}^m |c_{ij}|\, [ (v_j^\ep-w_j^\ep)-(v_i^\ep-w_i^\ep)]^2.
\end{align*}
Hence
\begin{multline}\notag
\ep(\varphi_i)_t
+D_pH_i(x,Dw_i^\ep)\cdot D\varphi_i
 +\sum_{j=1}^m c_{ij} \varphi_j\\
+\dfrac{1}{2}\sum_{j=1}^m |c_{ij}|\, [ (v_j^\ep-w_j^\ep)-(v_i^\ep-w_i^\ep)]^2
\le
\ep^4  \Delta\varphi_i- \ep^4 |D(v_i^\ep-w_i^\ep)|^2+C {|\ol{H}_\ep|}.
\end{multline}
Then \eqref{m-equ-1} follows immediately.

\medskip
\noindent
\textbf{Acknowledgments} 
The works of 
FC and DG were partially supported by CAMGSD-LARSys through FCT-Portugal 
and by grants PTDC/MAT/114397/2009,
UTAustin
/MAT/0057/2008, and UTA-CMU/MAT/0007/2009.  
The work of FC was partial 
supported by the UTAustin-Portugal partnership through the FCT post-doctoral fellowship
SFRH/BPD/51349/2011. 
The work of HM was partially supported by KAKENHI  
\#24840042, JSPS and Grant for Basic Science Research Projects from the 
Sumitomo Foundation.

\end{document}